\documentclass[11pt]{article}
\usepackage{amssymb}
\usepackage{amsmath}
\usepackage{amsfonts}
\usepackage{latexsym}
\usepackage{enumerate}
\usepackage{pspicture,epsfig,pstricks}
\usepackage{amsthm}
\usepackage{graphicx,mathtools}
\usepackage{mathtools}
\usepackage{graphicx}
\usepackage[font=small,labelfont=bf]{caption}

\usepackage{float}
\floatstyle{boxed}
\usepackage[english]{babel}
\usepackage{caption}

\usepackage[margin=1in]{geometry}
\usepackage{multirow,booktabs}
\usepackage[mathscr]{euscript}
\title{Betweenness Centrality in Some Classes of Graphs }
\author{\small {Sunil~Kumar R\thanks{ Under the Faculty Development Program of the University Grants Commission, Government of India.}}\\
        \small {Research~Scholar}\\
        \small{Department of Computer~Applications}\\
        \small{Cochin University of Science and Technology}\\
        \small{ e-mail:sunilstands@gmail.com} \and
        \and
             \small{Kannan Balakrishnan}\\
         \small{Associate~Professor}\\
           \small{Department of Computer~Applications}\\
           \small{Cochin University of Science and Technology}\\
           \small{ e-mail:mullayilkannan@gmail.com
                     }\and
                     \small{M.\,Jathavedan}\\
         \small{Professor Emeritus}\\
           \small{Department of Computer~Applications}\\
           \small{Cochin University of Science and Technology}\\
           \small{ e-mail:mjvedan@gmail.com}
                     }
           \date{ }           
             
\newtheorem{Theorem}{Theorem}[section]

\newtheorem{Corollary}{Corollary}[section]

\newtheorem{Ex}{Example}
\begin{document}
\maketitle

\begin{abstract} 
There are several centrality measures that have been introduced and studied for real world networks. They account for the different vertex characteristics that permit them to be ranked in order of importance in the network. Betweenness centrality is a measure of the influence of a vertex over the flow of information between every pair of vertices under the assumption that information primarily flows over the shortest path between them. In this paper we present betweenness centrality of some important classes of graphs.
\vspace{.2cm}
\end{abstract} 
 \section{Introduction}
Betweenness centrality plays an important role in analysis of social networks \cite{otte2002social,
  park2010social}, computer networks \cite{latora2007measure} and many other types of network data models \cite{estrada2006virtual,freeman1979centrality,koschutzki2008centrality,
  martin2010centrality,newman2001scientific,rubinov2010complex}. 
   
 In the case of communication networks the distance from other units is not the only important property of a unit. More important is which units lie on the shortest paths (geodesics) among pairs of other units. Such units have control over the flow of information in the network. Betweenness centrality is useful as a measure of the potential of a vertex for control of communication. Betweennes centrality \cite{bonacich1987power,borgatti2005centrality,borgatti2006graph,estrada2009communicability,
gago2013betweenness} indicates the betweenness of a vertex in a network and it measures the extent to which a vertex lies on the shortest paths between pairs of other vertices. In many real-world situations it has quite a significant role.

Determining betweenness is simple and straightforward when only one 
geodesic connects each pair of vertices, where the 
intermediate vertices can completely control communication between 
pairs of others. But when there are several geodesics connecting a pair of 
vertices, the situation becomes more complicated and the control of the intermediate vertices get fractionated. 
  \section{Background}
  The concept of betweenness centrality was first introduced by Bavelas in 1948 \cite{bavelasa}. The 
importance of the concept of vertex 
centrality is in the potential of a vertex for 
control of information flow in the network. 
Positions are viewed as structurally central to 
the degree that they stand between others and 
can therefore facilitate, impede or bias the 
transmission of messages. 
 Linton C. Freeman in his papers    \cite{freeman1977set,freeman1979centrality} classified betweenness centrality into three. 
 The three measures includes two indexes of vertex centrality - 
one based on counts and one on proportions - and one index of overall 
network or graph centralization. \\ 
\vspace{1mm}
\subsection{Betweenness Centrality of a Vertex}
 Betweenness centrality $C_B(v)$ for a vertex $v$ is defined as
$$C_B(v) = \sum_{s\neq{v}\neq{t}}{\frac{\sigma_{ st}(v)}{\sigma_{ st}}}$$ where $ \sigma_{st}$ is the number of shortest paths with vertices $s$ and $t$ as their end vertices, while $\sigma_{st}(v)$ is the 
number of those shortest paths that include vertex $v$ \cite{freeman1977set}.High centrality scores indicate that a vertex lies on a considerable fraction of shortest paths connecting pairs of vertices.

\begin{itemize}
\item Every pair of vertices in a connected graph provides a value lying in [0,1] to the betweenness centrality of all other vertices.
\item If there is only one geodesic joining a particular pair of vertices, then that pair provides a betweenness centrality 1 to each of its intermediate vertices and zero to all other vertices. For example in a path graph, a pair of vertices provides a betweenness centrality 1 to each of its interior vertices and zero to the exterior vertices. A pair of adjacent vertices always provides zero to all others.
\item If there are $k$ geodesics of length 2 joining a pair of vertices, then that pair of vertices provides a betweenness centrality $\frac{1}{k}$ to each of the intermediate vertices. 
\end{itemize}
 Freeman \cite{freeman1977set} proved that the maximum value taken by $C_B(v)$ is 
achieved only by the central vertex in a star as the central vertex lies  on the geodesic (which is unique) joining every pair of other vertices. In a star $S_n$ with $n$ vertices,  the betweenness centrality of the central vertex is therefore the number of such geodesics which is $\binom{n-1}{2}$. The betweenness centrality of each pendant vertex is zero since no pendant vertex lies in between any geodesic.
Again it can be seen that the betweenness centrality of any vertex in a complete graph $K_n$ is zero since no vertex lies in between any geodesic as every geodesic is of length 1. \\
\subsection{Relative Betweenness Centrality }
The betweenness centrality increases with the 
number of vertices in the network, so a normalized version is often considered with the centrality 
values scaled to between 0 and 1. Betweenness  centrality can be normalized by dividing $C_B (v)$ by its maximum value.  Among all graphs of $n$ vertices the central vertex of a star graph $S_n$ has the maximum value which is $\binom{n-1}{2}$. The relative betweenness centrality is therefore defined as
$$C_{B}{'}{(v)}=\frac{C_B(v)}{Max \,C_B(v)}=\frac{2C_B(v)}{(n-1)(n-2)}\qquad  0\leq C_B{^{'}}(v)\leq 1$$\\
\subsection{Betweenness Centrality of a Graph}
The betweenness centrality of a graph measures the tendency of a single vertex to be more central 
than all other vertices in the graph. It is based on differences between the centrality of the most central 
vertex and that of all others. Freeman \cite{freeman1977set} defined the betweenness centrality of a graph as the average difference between the measures of centrality of the most central vertex and that of all other vertices.\\

The betweenness centrality of a graph $G$ is defined as\\ 

$$C_B(G) = \frac{\sum_{i=1}^n[C_B  (v^*)-C_B (v_i)]}{Max \sum_{i=1}^n[C_B ( v^* )-C_B (v_i )])}$$ \\
where $C_B(v^*)$ is the largest value of $C_B(v_i)$ for any vertex $v_i$ in the given graph $G$ and\\
Max $\sum_{i=1}^n[C_B ( v^* )-C_B (v_i)] $ is the maximum possible sum of differences in  centrality for any graph of $n$ vertices which occur in star with the value  $n-1$ times $ C_B(v)$ of the central vertex. That is, $(n-1)\binom{n-1}{2}$.

Therefore the betweenness centrality of $G$ is defined as \\

$$C_B(G) = \frac{2\sum_{i=1}^n[C_B  (v^*)-C_B (v_i)]}{{(n-1)}^2 (n-2)}\quad or \quad C_B(G) =\frac{\sum_{i=1}^n[C_B {'} (v^*)-C_B{'} (v_i)]}{{(n-1)}}$$
The index, $C_B(G)$, determines the degree to which 
$C_B(v^*)$ exceeds the centrality of all other vertices in $G$. Since $C_B(G)$ is 
the ratio of an observed sum of differences to its maximum value, it will vary 
between 0 and 1. 
$C_B(G)$ = 0 if and only if all $C_B(v_i)$ are equal, and $C_B(G)$ = 1 if and only if one 
vertex $v*$, completely dominates the network with respect to centrality. Freeman showed that all of these measures agree in assigning the highest centrality index to the star graph and the lowest to the  complete graph.  

\begin{table}[H]	
\begin{center}
  \begin{tabular}{| c | c | c |c| }
    \hline
$G$ & $C_B(v)$ & $C_B{'}(v)$ & $C_B(G)$\\
	\hline
	$S_n$& 
$  
\begin{dcases*}\binom{n-1}{2}  & for central vertex\\
        \hspace{.6cm}0 & for other vertices
        \end{dcases*}
$&$\begin{dcases*}
        1  & for central vertex\\
        0 & for other vertices
        \end{dcases*}$&1\\ \hline
 $K_n$&0&0&0\\ \hline       
\end{tabular}
 \caption{Graphs showing extreme betweenness}
\end{center}
\end{table}
In this paper we present the  betweenness centrality measures in some important classes of graphs. 
\section{Betweenness centrality of some classes of graphs}

\subsection{Betweenness centrality of vertices in wheels }
\begin{Theorem}
The betweenness centrality of a vertex $v$ in a wheel graph $W_n,~n>5$ is given by 

\[
 C_{B}{(v)} =  \begin{dcases*}
       \frac{(n-1)(n-5)}{2}& if $v$ is the central vertex\\
       \hspace{.5cm} \frac{1}{2}&  otherwise 
        \end{dcases*}
\]
\end{Theorem}
\vspace{0cm}
\begin{proof}

In the wheel graph $W_n$ the central vertex is adjacent to each vertex of the cycle $C_{n-1}$. 
 When $n>5$ consider the central vertex. On $C_{n-1}$ each pair of adjacent vertices  contributes $0$, each pair of alternate vertices contributes $\frac{1}{2}$ and all other pairs contribute  centrality $1$ to the central vertex.  Since there are $n-1$ vertices on $C_{n-1}$, there exists $n-1$ adjacent pairs, $n-1$ alternate pairs and $\binom{n-1}{2}-2(n-1)=\frac{(n-1)(n-6)}{2}$ other pairs. Therefore the central vertex has betweenness centrality $\frac{1}{2}(n-1)+ 1\frac{(n-1)(n-6)}{2}=\frac{(n-1)(n-5)}{2}$. Now for any vertex on $C_{n-1}$, there are two geodesics joining its adjacent vertices on $C_{n-1}$, one of which passing through it. Therefore its betweenness centrality is $\frac{1}{2}$.
\end{proof}
\noindent \textit{Note}\\
It can be seen easily that $C_B(v)=0$ for every vertex $v$ in $W_4$   and  
\[
 C_{B}{(v)} =  \begin{dcases*}
       \frac{2}{3}& if $v$ is the central vertex\\
       \hspace{0cm} \frac{1}{3}&  otherwise 
        \end{dcases*}
\]
in $W_5$.\\\\
The relative centrality and graph centrality are as follows
\[
 C_{B}{'}{(v)} = \frac{2C_B(v)}{(n-1)(n-2)}= \begin{dcases*}
       \frac{(n-5)}{n-2}& if $v$ is the central vertex\\
       \frac{1}{(n-1)(n-2)}& otherwise
        \end{dcases*}
\]

$$C_B(W_n)=\frac{\sum_{i=1}^n[C_B {'} (v^*)-C_B{'} (v_i)]}{{(n-1)}}=\frac{n^2-6n+4}{(n-1)(n-2)}$$
\subsection{Betweenness centrality of vertices in the graph $\mathbf{K_n-e}$}

\begin{Theorem}
Let $K_n$ be a complete graph on $n$ vertices and $e=(v_i,v_j)$ be an edge of it. Then the betweenness centrality of vertices in $K_n-e$ is given by
\[
 C_{B}{(v)} =  \begin{dcases*}
       \frac{1}{n-2}& if $v\neq  v_i,v_j$ \\
       \hspace{0cm}0 &  otherwise 
        \end{dcases*}
\]
\end{Theorem}

\begin{proof}
Suppose the edge $(v_i,v_j)$ is removed from $K_n$. Now $v_i$ and $v_j$ can be joined by means of any of the remaining $n-2$ vertices. Thus there are $n-2$ geodesics joining $v_i$ and $v_j$ each containing exactly one vertex as intermediary. This provides a betweenness centrality $\frac{1}{n-2}$ to each of the $n-2$ vertices. Again $v_i$ and $v_j$ do not lie in between  any geodesics and therefore their betweenness centralities zero. 
\end{proof}
The relative centrality and graph centrality are as follows
\[
 C_{B}{'}{(v)} = \frac{2C_B(v)}{(n-1)(n-2)}= \begin{dcases*}
       \frac{2}{(n-1)(n-2)^2}& $v\neq  v_i,v_j$\\
       0& otherwise
        \end{dcases*}
\]

$$C_B(G)=\frac{\sum_{i=1}^n[C_B {'} (v^*)-C_B{'} (v_i)]}{{(n-1)}}=\frac{4}{(n-1)^2(n-2)^2}$$
\subsection{Betweenness centrality of vertices in complete bipartite graphs}
\begin{Theorem}
The betweenness centrality of a vertex in a complete bipartite graph $K_{m,n}$  is given by

\[
 C_{B}{(v)} =  \begin{dcases*}
       \frac{1}{m}\binom {n}{2} & if deg(v)=n\\
        \frac{1}{n}\binom {m}{2} & if deg(v)=m
        \end{dcases*}
\]
\vspace{1cm}
\end{Theorem}
\begin{proof}
Consider a complete bipartite graph $K_{m,n}$ with a bipartition $\{U,W\}$ where $U=\{u_1,u_2,...,u_m\}$ and $W=\{w_1,w_2,\dots,w_n\}$. The distance between any two vertices in $U$ (or in $W$) is 2. Consider a vertex $u\in U$. Now any pair of vertices in $W$ contributes a betweenness centrality $\frac{1}{m}$ to $u$. Since there are $\binom{n}{2}$ pairs of vertices in $W$. $C_B(u)=\frac{1}{m}\binom {n}{2}$.
In a similar way it can be shown that for any vertex $w$ in $W$,\,  $C_B(w)=\frac{1}{n}\binom {m}{2}$.
\end{proof}
The relative centrality and graph centrality are as follows
\[
 C_{B}{'}{(v)}=\frac{2C_B(v)}{(m+n-1)(m+n-2)}= \begin{dcases*}
       \frac{2}{(m+n-1)(m+n-2)}\times\frac{1}{m}\binom{n}{2}& if deg(v)=n\\
       \frac{2}{(m+n-1)(m+n-2)}\times\frac{1}{n}  		\binom{m}{2}& if deg(v)=m
        \end{dcases*}
\]
\[
 C_{B}{(K_{m,n})}=\frac{\sum_{i=1}^n[C_B {'} (v^*)-C_B{'} (v_i)]}{{(m+n-1)}}= \begin{dcases*}
       \frac{m^3-n^3-(m^2-n^2)}{n(m+n-1)^2(m+n-2)}& if $m > n$\\
       \frac{n^2(n-1)-m^2(m-1)}{m(m+n-1)^2(m+n-2)}& if $n > m$
        \end{dcases*}
\]
\subsection{Betweenness centrality of vertices in cocktail party graphs}
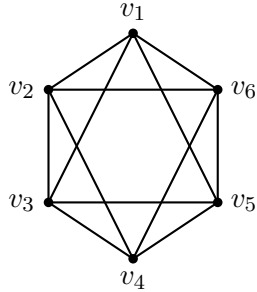
\begin{figure}[H]
\psset{unit=.75}\label{fig7}
\begin{picture}(10000,150)(-250,-60) 
\psdot(0,2)
\uput[u](0,2){$v_1$}
\psdot(-1.5,1)
\uput[l](-1.5,1){$v_2$}
\psdot(-1.5,-1)
\uput[l](-1.5,-1){$v_3$}
\psdot(0,-2)
\uput[d](0,-2){$v_4$}
\psdot(1.5,-1)
\uput[r](1.5,-1){$v_5$}
\psdot(1.5,1)
\uput[r](1.5,1){$v_6$}

\psline(0,2)(-1.5,1)
\psline(-1.5,1)(-1.5,-1)
\psline(-1.5,-1)(0,-2)
\psline(0,-2)(1.5,-1)
\psline(1.5,-1)(1.5,1)
\psline(1.5,1)(0,2)

\psline(0,2)(-1.5,-1)
\psline(0,2)(1.5,-1)
\psline(-1.5,1)(0,-2)
\psline(-1.5,1)(1.5,1)

\psline(-1.5,-1)(1.5,-1)

\psline(0,-2)(1.5,1)

\end{picture}

\caption {Cocktail party graph CP(3)}
\centering
\end{figure}
The cocktail party graph $CP(n)$  \cite{cvetkovic̀1981generalized} is a unique regular graph
of degree $2n-2$ on $2n$ vertices. It is obtained from $K_{2n}$ by deleting a perfect
matching. The cocktail party graph of order $n$, is a complete $n$-partite graph with 2 vertices in each partition set. It is the graph complement of the ladder rung graph $L_n\prime$ which is the graph union of n copies of the path graph $P_2$ and the dual graph of the hypercube $Q_n$ \cite{biggs1993algebraic}.

\begin{Theorem}
The betweenness centrality of each vertex of a cocktail party graph of order $2n$ is $\frac{1}{2}$.
\end{Theorem}
\begin{proof}
Let the cocktail party graph $CP(n)$ be obtained from the complete graph $K_{2n}$ with vertices \\$\{v_1,\dots,v_n,v_{n+1},\dots,v_{2n}\}$ by deleting a  perfect matching $\{(v_1,v_{n+1}),(v_2,v_{n+2}),...,(v_n,v_{2n}\}$. Now for each pair $(v_i,v_{n+i})$ there is a geodesic path of length 2 passing through each of the other $2n-2$ vertices. Thus for any particular vertex, there are $n-1$ pairs of vertices of the above matching not containing that vertex giving a betweenness centrality $\frac{1}{2n-2}$ to that vertex. Therefore the betweenness centrality of any vertex is given by $\frac{n-1}{2n-2}=\frac{1}{2}$.
\end{proof}
The relative centrality and graph centrality are as follows\\ $$C_{B}{'}{(v)}=\frac{2C_B(v)}{(2n-1)(2n-2)}=\frac{1}{(2n-1)(2n-2)}$$,     

$$C_B(G)=0$$
\subsection{Betweenness centrality of vertices in crown graphs}

The crown graph \cite{biggs1993algebraic} is the unique $n-1$ regular graph with $2n$ vertices obtained from a complete bipartite graph $K_{n,n}$ by deleting a perfect
matching.
A crown graph on 2$n$ vertices can be viewed as an undirected graph with two sets of vertices
$u_i$ and $v_i$ and with an edge from $u_i$ to $v_j$ whenever $i\neq j$.
It is the graph complement of the ladder graph $L_{2n}$. The crown graph is a distance-transitive graph.
\begin{Theorem}
The betweenness centrality of each vertex of a crown graph of order $2n$ is $\frac{n+1}{2}$.
\end{Theorem}

\begin{figure}[H]
\psset{unit=.75}\label{fig7}
\begin{picture}(10000,60)(-150,30)
\psdot(2.5,4)
\uput[u](2.5,4){$u_1$}
\psdot(4,4)
\uput[u](4,4){$u_2$}
\psdot(5.5,4)
\uput[u](5.5,4){$u_3$}
\psdot(7,4)
\uput[u](7,4){$u_4$}
\psdot(2.5,2)
\uput[d](2.5,2){$v_1$}
\psdot(4,2)
\uput[d](4,2){$v_2$}
\psdot(5.5,2)
\uput[d](5.5,2){$v_3$}
\psdot(7,2)
\uput[d](7,2){$v_4$}
\psline(2.5,4)(4,2)
\psline(2.5,4)(5.5,2)
\psline(2.5,4)(7,2)
\psline(4,4)(2.5,2)
\psline(4,4)(5.5,2)
\psline(4,4)(7,2)
\psline(5.5,4)(2.5,2)
\psline(5.5,4)(4,2)
\psline(5.5,4)(7,2)
\psline(7,4)(4,2)
\psline(7,4)(5.5,2)
\psline(7,4)(2.5,2)
\end{picture}
\caption{Crown graph with 8-vertices}
\centering
\end{figure}
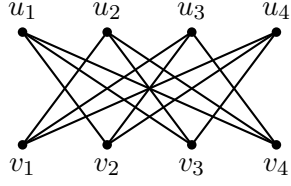
\begin{proof}
Let the crown graph be the complete bipartite graph $K_{n,n}$ with vertices $\{u_1,\dots,u_n,v_1,\dots,v_n\}$ minus a perfect matching $\{(u_1,v_1),(u_2,v_2),...,(u_n,v_n)\}$.
Consider any vertex say $u_1$. Now for each $(u_i,v_i),i\neq1$ each pair other than $(u_1,v_1)$ there are $n-2$ paths of length 3 passing through $u_1$ out of $(n-1)(n-2)$ paths joining $u_i$ and $v_i$. Since there are $n-1$ such pairs, it gives $v_1$ a betweenness centrality  $(n-1)\times\frac{n-2}{(n-1)(n-2)}=1$. Again for each pair from $\{v_2,v_3,v_4,...,v_n\}$ there exists exactly one path passing through $v_1$ out of $n-2$ paths. It gives the betweenness centrality $\frac{n-2}{n-2}+ \frac{n-3}{n-2}+...+\frac{1}{n-2}=\frac{1}{n-2}[(n-2)+(n-3)+...+1]=\frac{n-1}{2}$. Therefore the betweenness centrality of $v_1$ is given by $1+\frac{n-1}{2}=\frac{n+1}{2}$. Since the graph is vertex transitive, the betweenness centrality of any vertex is given by $\frac{n+1}{2}$.

\end{proof}
The relative centrality and graph centrality are as follows\\ $$C_{B}{'}{(v)}=\frac{2C_B(v)}{(2n-1)(2n-2)}=\frac{n+1}{(2n-1)(2n-2)}$$,     
$$C_B(G)=0$$
\subsection{Betweenness centrality of vertices in paths}
\begin{Theorem}
The betweenness centrality of any vertex in a path graph is the product of the number of vertices on either side of that vertex in the path.
\end{Theorem}
\begin{proof}
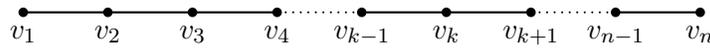
\begin{figure}[H]
\psset{unit=.75}\label{fig7}
\begin{picture}(10000,50)(-50,50)
\psdot(2.5,4)
\uput[d](2.5,4){$v_1$}
\psdot(4,4)
\uput[d](4,4){$v_2$}
\psdot(5.5,4)
\uput[d](5.5,4){$v_3$}
\psdot(7,4)
\uput[d](7,4){$v_4$}
\psdot(8.5,4)
\uput[d](8.5,4){$v_{k-1}$}
\psdot(10,4)
\uput[d](10,4){$v_{k}$}
\psdot(11.5,4)
\uput[d](11.5,4){$v_{k+1}$}
\psdot(13,4)
\uput[d](13,4){$v_{n-1}$}
\psdot(14.5,4)
\uput[d](14.5,4){$v_{n}$}
\psline(2.5,4)(7,4)
\psline[linestyle=dotted,dotsep=2pt](7,4)(8.5,4)
\psline(8.5,4)(11.5,4)
\psline[linestyle=dotted,dotsep=2pt](11.5,4)(13,4)
\psline(13,4)(14.5,4)
\end{picture}
\caption{Path graph $P_n$}
\centering
\end{figure}
Consider a path graph $P_n$ of $n$ vertices $\{v_1,v_2,...,v_n\}$. Take a vertex $v_k$ in $P_n$. Then there are $k-1$ vertices on one side and $n-k$ vertices on the other side of $v_k$. Consequently there are $(k-1)\times (n-k)$ number of geodesic paths containing $v_k$. Hence $C_B(v_k)=(k-1)(n-k)$.
\end{proof}
Note that by symmetry, vertices at equal distance away from both the ends have the same centrality and it is maximum at the central vertex and minimum at the end vertex.       
        \[
 Max~ C_{B}{(v_k)} =  \begin{dcases*}
       \frac{n(n-2)}{4} & when n is even\\
       \frac{(n-1)^2}{4} & when n is odd\\
        \end{dcases*}\]
\\

Relative centrality of any vertex $v_k$ is given by $$C_B{'}(v_k)=\frac{2C_B(v_k)}{(n-1)(n-2)}=\frac{2(k-1)n-k)}{(n-1)(n-2)}$$\\
\begin{Corollary}

Graph centrality of $P_n$ is given by\\
\[
 C_{B}{(P_n)} =  \begin{dcases*}
       \frac{n(n+1)}{6(n-1)(n-2)} & if n is odd\\
        \frac{n(n+2)}{6(n-1)^2} & if n is even
        \end{dcases*}\]
        \end{Corollary}

\begin{proof}
When $n$ is even, \text{by definition}
\begin{align*}
C_B(P_n) &= \frac{2}{(n-1)^2(n-2)}\sum_{i=1}^n[C_B  (v^*)-C_B (v_i)]\\
&= \frac{4}{(n-1)^2(n-2)} \Bigr\{\Big[\frac{n(n-2)}{4}-0\Big]+\Big[\frac{n(n-2)}{4}-1.(n-2)\Big]+\ldots+\Big[\frac{n(n-2)}{4}-\big(\frac{n-4
}{2}\big)\big(n-\frac{n-2}{2}\big)\Big]\Bigr\}\\
&=\frac{4}{(n-1)^2(n-2)}\Bigr\{\Big[\frac{n(n-2)}{4}\times\frac{n-2}{2}\Big]-\Big[1(n-2)+2(n-3)+\dots+\big( \frac{n-4}{2}\big)\big(n-\frac{n-2}{2}\big)\Big]\Bigr\}\\
&=\frac{4}{(n-1)^2(n-2)}\Bigr\{\frac{n(n-2)^2}{8}-n\sum_{k=1}^{\frac{n-4}{2}}k+\sum_{k=1}^{\frac{n-4}{2}}k(k+1)\Bigr\}\\
&=\frac{4}{(n-1)^2(n-2)}\Bigr\{\frac{n(n-2)^2}{8}-\frac{n(n-2)(n-4)}{8}+\frac{(n-2)(n-4)}{8}+\frac{(n-2)(n-3)(n-4)}{24}\Bigr\}\\
&=\frac{n(n+2)}{6(n-1)^2}
\end{align*}\\

When $n$ is odd, \text{by definition}
\begin{align*}
C_B(P_n) &= \frac{2}{(n-1)^2(n-2)}\sum_{i=1}^n[C_B  (v^*)-C_B (v_i)]\\
&= \frac{4}{(n-1)^2(n-2)} \Bigr\{\Big[\frac{(n-1)^2}{4}-0\Big]+\Big[\frac{(n-1)^2}{4}-1.(n-2)\Big]+\ldots+\Big[\frac{(n-1)^2}{4}-\big(\frac{n-3}{2}\big)\big(n-\frac{n-1}{2}\big)\Big]\Bigr\}\\
&=\frac{4}{(n-1)^2(n-2)}\Bigr\{\Big[\frac{(n-1)^2}{4}\times\frac{n-1}{2}\Big]-\Big[\Big(1(n-2)+2(n-3)+\dots+\big( \frac{n-3}{2}\big)\big(n-\frac{n-1}{2}\big)\Big]\Bigr\}\\
&=\frac{4}{(n-1)^2(n-2)}\Bigr\{\frac{(n-1)^3}{8}-n\sum_{k=1}^{\frac{n-3}{2}}k +\sum_{k=1}^{\frac{n-3}{2}}k(k+1)\Bigr\}\\
&= \frac{4}{(n-1)^2(n-2)} \Bigr\{\frac{(n-1)^3}{8}-\frac{n(n-1)(n-3)}{8}+ \frac{(n-1)(n+1)(n-3)}{24}\Bigr\}\\
&=\frac{n(n+1)}{6(n-1)(n-2)}
\end{align*}
\end{proof}
\subsection{Betweenness centrality of vertices in ladder graphs}

The ladder graph $L_n$ \cite{hosoya1993matching,noy2004recursively} is a planar undirected graph with $2n$ vertices and $n+2(n-1)$ edges.
It can be defined as the cartesian product $P_2\times P_n$.\\ 
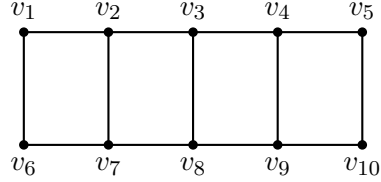
\begin{figure}[h]
\psset{unit=.75}\label{fig7}
\begin{picture}(10000,100)(-110,30)
\psdot(2.5,4)
\uput[u](2.5,4){$v_1$}
\psdot(4,4)
\uput[u](4,4){$v_2$}
\psdot(5.5,4)
\uput[u](5.5,4){$v_3$}
\psdot(7,4)
\uput[u](7,4){$v_{4}$}
\psdot(8.5,4)
\uput[u](8.5,4){$v_{5}$}

\psline(2.5,4)(8.5,4)

\psdot(2.5,2)
\uput[d](2.5,2){$v_{6}$}
\psdot(4,2)
\uput[d](4,2){$v_{7}$}
\psdot(5.5,2)
\uput[d](5.5,2){$v_{8}$}
\psdot(7,2)
\uput[d](7,2){$v_{9}$}
\psdot(8.5,2)
\uput[d](8.5,2){$v_{10}$}
\psline(2.5,2)(8.5,2)

\psline(2.5,4)(2.5,2)
\psline(4,4)(4,2)
\psline(5.5,4)(5.5,2)
\psline(7,4)(7,2)
\psline(8.5,4)(8.5,2)
\end{picture}

\caption{Ladder graph $L_5$}
\centering
\end{figure}
 \begin{Theorem}
 The betweenness centrality of a vertex in a Ladder graph $L_n$ is given by
 $$C_B(v_k)= (k-1)(n-k)+\sum_{j=0}^{k-1}\sum_{i=1}^{n-k} \frac{k-j}{k-j+i}+\sum_{j=0}^{k-2}\sum_{i=0}^{n-k} \frac{i+1}{k-j+i},\;1\leq k\leq n$$
 \end{Theorem}
 \begin{proof}
 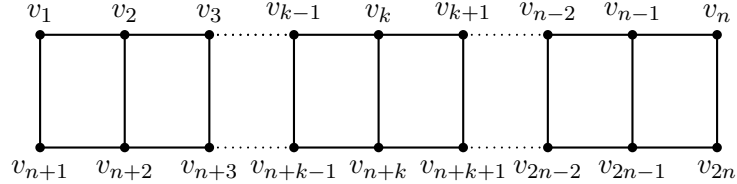
\begin{figure}[H]
\psset{unit=.75}\label{fig7}
\begin{picture}(10000,100)(-70,30)
\psdot(2.5,4)
\uput[u](2.5,4){$v_1$}
\psdot(4,4)
\uput[u](4,4){$v_2$}
\psdot(5.5,4)
\uput[u](5.5,4){$v_3$}
\psdot(7,4)
\uput[u](7,4){$v_{k-1}$}
\psdot(8.5,4)
\uput[u](8.5,4){$v_{k}$}
\psdot(10,4)
\uput[u](10,4){$v_{k+1}$}
\psdot(11.5,4)
\uput[u](11.5,4){$v_{n-2}$}
\psdot(13,4)
\uput[u](13,4){$v_{n-1}$}
\psdot(14.5,4)
\uput[u](14.5,4){$v_{n}$}
\psline(2.5,4)(5.5,4)
\psline[linestyle=dotted,dotsep=2pt](5.5,4)(7,4)
\psline[linestyle=dotted,dotsep=2pt](5.5,4)(7,4)
\psline(7,4)(10,4)
\psline(11.5,4)(14.5,4)
\psline[linestyle=dotted,dotsep=2pt](10,4)(11.5,4)
\psdot(2.5,2)
\uput[d](2.5,2){$v_{n+1}$}
\psdot(4,2)
\uput[d](4,2){$v_{n+2}$}
\psdot(5.5,2)
\uput[d](5.5,2){$v_{n+3}$}
\psdot(7,2)
\uput[d](7,2){$v_{n+k-1}$}
\psdot(8.5,2)
\uput[d](8.5,2){$v_{n+k}$}
\psline(2.5,2)(4,2)
\psline(4,2)(5.5,2)
\psline[linestyle=dotted,dotsep=2pt](10,2)(11.5,2)
\psline[linestyle=dotted,dotsep=2pt](5.5,2)(7,2)
\psdot(10,2)
\uput[d](10,2){$v_{n+k+1}$}
\psdot(11.5,2)
\uput[d](11.5,2){$v_{2n-2}$}
\psdot(13,2)
\uput[d](13,2){$v_{2n-1}$}
\psdot(14.5,2)
\uput[d](14.5,2){$v_{2n}$}
\psline(7,2)(10,2)
\psline(14.5,4)(14.5,2)
\psline(11.5,2)(14.5,2)
\psline(2.5,4)(2.5,2)
\psline(4,4)(4,2)
\psline(5.5,4)(5.5,2)
\psline(7,4)(7,2)
\psline(8.5,4)(8.5,2)
\psline(10,4)(10,2)
\psline(11.5,4)(11.5,2)
\psline(13,4)(13,2)
\end{picture}
\caption{Ladder graph $L_n$}
\end{figure}
 By symmetry, let $v_k$ be any vertex such that $1\leq k\leq\frac{n+1}{2}$. Consider the paths (in fig 3.5) from upper left vertices $\{v_1, \dots,v_{k-1}\}$ to upper right vertices $\{v_{k+1},\dots, v_n\}$ which gives the betweenness centrality $$(k-1)(n-k)\hspace*{7cm}(1)$$ Now consider the paths from lower left vertices $\{v_{n+1},\dots,v_{n+k}\}$ to the upper right vertices $\{v_{k+1},\dots,v_n\}$ of $v_k$. It gives the betweenness centrality
$$k\left\{\frac{1}{k+1}+\frac{1}{k+2}+\dots+\frac{1}{n}\right\}+(k-1)\left\{\frac{1}{k}+\frac{1}{k+1}+\dots+\frac{1}{n-1}\right\}+\dots+\left\{\frac{1}{2}+\frac{1}{3}+\dots+\frac{1}{n-(k-1)}\right\}$$\\
$$ 
=\sum_{j=0}^{k-1}\sum_{i=1}^{n-k} \frac{k-j}{k-j+i}\hspace*{6cm}(2)$$ Consider the paths from upper left vertices $\{v_{1}, \dots,v_{k-1}\}$ to the lower right vertices $\{v_{n+k}, \dots,v_{2n}\}$ of $v_k$. It gives the betweenness centrality $$\left\{\frac{1}{k}+\frac{2}{k+1}+\dots+\frac{n-(k-1)}{n}\right\}
 +\left\{\frac{1}{k-1}+\frac{2}{k}+\dots+\frac{n-(k-1)}{n-1}\right\}+\dots+ \left\{\frac{1}{2}+\frac{2}{3}+\dots+\frac{n-(k-1)}{n-(k-2)}\right\}$$\\
  $$=\sum_{j=0}^{k-2}\sum_{i=0}^{n-k} \frac{i+1}{k-j+i}\hspace*{6cm}(3)$$
  The above three equations when combined get the result.
 \end{proof}
 \subsection{Betweenness centrality of vertices in trees}
In a tree, there is exactly one path between any two vertices. Therefore the betweenness centrality of a vertex is the number of paths passing through that vertex. A branch at a vertex $v$ of a tree $T$ is a maximal subtree containing $v$ as an end vertex. The number of branches at $v$ is $deg(v)$.  \\

\begin{Theorem}
The betweenness centrality $C_B(v)$ of a vertex $v$ in a tree $T$ is given by \\
$$\mathcal{C}(n_1,n_2,\dots,n_k)=\sum_{i<j}n_in_j$$
where the arguments $n_i$ denotes the number of vertices of the branches at $v$ excluding $v$, taken in any order.
\end{Theorem}
\begin{proof}
Consider a vertex $v$ in a tree $T$. Let there are $k$ branches with number of vertices $n_1,n_2,\dots,n_k$ excluding $v$. The betweenness centrality of $v$ in $T$ is the total number of paths passing through $v$. Since all the branches have only one  vertex $v$ in common, excluding $v$, every path joining a pair of vertices of different branches passes through $v$. Thus the total number of such pairs gives the betweenness centrality of $v$. Hence $\mathcal{C}=\sum_{i<j}n_in_j$
\end{proof}
\begin{Ex}
Consider the tree given below
\end{Ex}
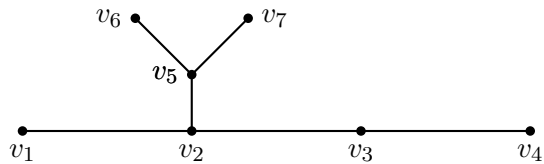
\begin{figure}[H]
\psset{unit=.75}\label{fig7}
\centering
\begin{picture}(10000,50)(-150,20)
\psdot(0,0)
\uput[d](0,0){$v_1$}
\psdot(3,0)
\uput[d](3,0){$v_2$}
\psdot(6,0)
\uput[d](6,0){$v_3$}
\psdot(9,0)
\uput[d](9,0){$v_4$}
\psdot(3,1)
\uput[l](3,1){$v_5$}
\psdot(2,2)
\uput[l](2,2){$v_6$}
\uput[l](3,1){$v_5$}
\psdot(4,2)
\uput[r](4,2){$v_7$}
\psline(0,0)(9,0)
\psline(3,0)(3,1)
\psline(3,1)(2,2)
\psline(3,1)(4,2)
\end{picture}\vspace{1cm}
\caption{A tree with 7 vertices}
\end{figure}
\begin{flushleft}
    Here
\end{flushleft}
$C_B(v_1)=C_B(v_4)=C_B(v_6)=C_B(v_7)=\mathcal{C}(6)=0$\\
$C_B(v_2)=\mathcal{C}(1,3,2)=11$\\
$C_B(v_3)=\mathcal{C}(5,1)=5$\\
$C_B(v_5)=\mathcal{C}(1,1,4)=9$\\

\begin{Ex}
The following table gives the possible values for the betweenness centrality of a vertex in a tree of 9 vertices.
\end{Ex}
We consider the different possible combinations of the arguments in $\mathcal{C}$ so that the sum of arguments is $8$. \\
\begin{table}[H]	
\begin{center}
  \begin{tabular}{| c | l | c | }
    \hline
    No.of args. & Possible combinations & Values \\ \hline
    8 & $\mathcal{C}(1,1,1,1,1,1,1,1)$ & 28 \\ \hline
    7 & $\mathcal{C}(2,1,1,1,1,1,1)$& 27 \\ \hline 
    6 & $\mathcal{C}(2,2,1,1,1,1)$& 26 \\ \cline{2-3}
      & $\mathcal{C}(3,1,1,1,1,1)$ & 25\\ \hline 
    5 & $\mathcal{C}(2,2,2,1,1)$ & 25 \\ \cline{2-3}
      &  $\mathcal{C}(3,2,1,1,1)$ & 24 \\ \cline{2-3}
      & $\mathcal{C}(4,1,1,1,1)$ & 22 \\ \hline    4 & $\mathcal{C}(2,2,2,2)$  & 24 \\ \cline{2-3}
      &  $\mathcal{C}(3,2,2,1)$ & 23 \\ \cline{2-3}
      & $\mathcal{C}(3,3,1,1)$ & 22 \\ \cline{2-3}
      & $\mathcal{C}(5,1,1,1)$ & 18 \\ \cline{2-3}
      &  $\mathcal{C}(4,2,1,1)$ & 21 \\ \hline 
    3 & $\mathcal{C}(3,3,2)$ & 21 \\ \cline{2-3}
      &	$\mathcal{C}(4,2,2)$ & 20 \\ \cline{2-3}
      & $\mathcal{C}(4,3,1)$ & 19 \\ \cline{2-3}
      & $\mathcal{C}(5,2,1)$ & 17 \\ \cline{2-3}
      & $\mathcal{C}(6,1,1)$ & 13 \\ \hline 
    2 & $\mathcal{C}(4,4)$ & 16 \\ \cline{2-3}
      &  $\mathcal{C}(5,3)$ & 15\\ \cline{2-3}
      &$\mathcal{C}(6,2)$ & 12\\ \cline{2-3}
      &$\mathcal{C}(7,1)$ & 7\\ \hline 
    1 &$\mathcal{C}(8)$ & 0\\ \hline
   
  \end{tabular}
 \caption{Possible values for betweenness centrality in a tree of 9 vertices}
\end{center}
\end{table}
\subsection{Betweenness centrality of vertices in cycles }
\begin{Theorem}
The betweenness centrality of a vertex in a cycle $C_n$ is given by 

\[
 C_{B}{(v)} =  \begin{dcases*}
       \frac{(n-2)^2}{8}& if $n$ is even\\
        \frac{(n-1)(n-3)}{8}& if $n$ is odd
        \end{dcases*}
\]

\end{Theorem}

\begin{proof}
Case$1$: When $n$ is even\\
\begin{figure}[H]
\psset{unit=.75}\label{fig7}
\begin{picture}(10000,100)(-70,30)
\psdot(3,3)
\uput[l](3,3){$v_1$}
\psdot(4,4)
\uput[u](4,4){$v_2$}
\psdot(5.5,4)
\uput[u](5.5,4){$v_3$}
\psdot(7,4)
\uput[u](7,4){$v_4$}
\psdot(8.5,4)
\uput[u](8.5,4){$v_5$}
\psdot(10,4)
\uput[u](10,4){$v_{k-3}$}
\psdot(11.5,4)
\uput[u](11.5,4){$v_{k-2}$}
\psdot(13,4)
\uput[u](13,4){$v_{k-1}$}
\psdot(14.5,4)
\uput[u](14.5,4){$v_{k}$}
\psline(3,3)(4,4)
\psline(4,4)(8.5,4)
\psline(10,4)(14.5,4)
\psline[linestyle=dotted,dotsep=2pt](8.5,4)(10,4)
\psdot(15.5,3)
\uput[r](15.5,3){$v_{k+1}$}
\psline(14.5,4)(15.5,3)
\psdot(4,2)
\uput[d](4,2){$v_{2k}$}
\psdot(5.5,2)
\uput[d](5.5,2){$v_{2k-1}$}
\psdot(7,2)
\uput[d](7,2){$v_{2k-2}$}
\psdot(8.5,2)
\uput[d](8.5,2){$v_{2k-3}$}
\psline(3,3)(4,2)
\psline(4,2)(8.5,2)
\psline[linestyle=dotted,dotsep=2pt](8.5,2)(10,2)
\psdot(10,2)
\uput[d](10,2){$v_{k+5}$}
\psdot(11.5,2)
\uput[d](11.5,2){$v_{k+4}$}
\psdot(13,2)
\uput[d](13,2){$v_{k+3}$}
\psdot(14.5,2)
\uput[d](14.5,2){$v_{k+2}$}
\psline(15.5,3)(14.5,2)
\psline(10,2)(14.5,2)
\end{picture}
\caption{Even Cycle with 2k vertices}
\end{figure}
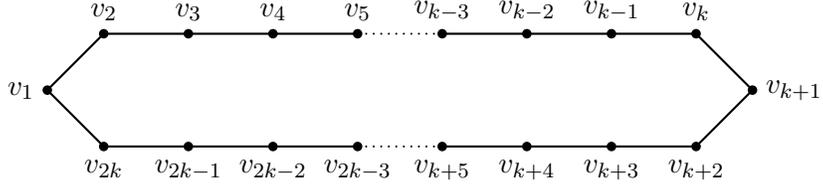
Let $n=2k,k\in \mathbb{Z^+}$ and $C_n=(v_1,v_2,...,v_{2k})$ be the given cycle. Consider a vertex $v_1$. Then $v_{k+1}$ is its antipodal vertex and there is no geodesic path from $v_{k+1}$ to any other vertex passing through $v_1$. Hence we omit the pair $(v_1,v_{k+1})$. Consider other pairs of antipodal vertices  $(v_i,v_{k+i})$ for $i=2,3,\dots, k$. For each pair of these antipodal vertices there exists two paths of the same length $k$ and one of them contains $v_1$. Thus each pair contributes $\frac{1}{2}$ to the centrality of $v_1$ and which gives a total of $\frac{1}{2}(k-1)$. Now consider all paths of length less than $k$ containing $v_1$. There are $k-i$ paths joining $v_i$ to vertices from $v_{2k}$ to $v_{k+i+1}$ passing through $v_1$ for $i=2,3,\dots,k-1$ and each contributes centrality 1 to $v_1$ giving a total $\sum_{i=2}^{k-1}(k-i)=\frac{(k-1)(k-2)}{2}$. Therefore the betweenness centrality of $v_1$ is  $\frac{1}{2}(k-1)+\frac{(k-1)(k-2)}{2}=\frac{1}{2}(k-1)^2=\frac{1}{8}(n-2)^2$. Since $C_n$ is vertex transitive, the betweenness centrality of any vertex is given by $\frac{1}{8}(n-2)^2$. \\ 

Case$2$: When $n$ is odd\\
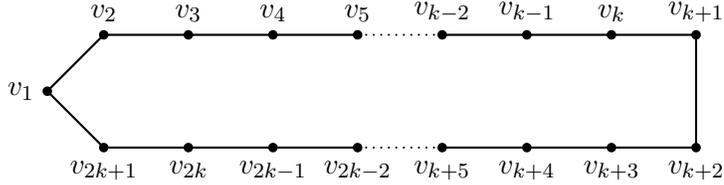
\begin{figure}[H]
\psset{unit=.75}\label{fig7}
\begin{picture}(10000,100)(-70,30)
\psdot(3,3)
\uput[l](3,3){$v_1$}
\psdot(4,4)
\uput[u](4,4){$v_2$}
\psdot(5.5,4)
\uput[u](5.5,4){$v_3$}
\psdot(7,4)
\uput[u](7,4){$v_4$}
\psdot(8.5,4)
\uput[u](8.5,4){$v_5$}
\psdot(10,4)
\uput[u](10,4){$v_{k-2}$}
\psdot(11.5,4)
\uput[u](11.5,4){$v_{k-1}$}
\psdot(13,4)
\uput[u](13,4){$v_{k}$}
\psdot(14.5,4)
\uput[u](14.5,4){$v_{k+1}$}
\psline(3,3)(4,4)
\psline(4,4)(8.5,4)
\psline(10,4)(14.5,4)
\psline[linestyle=dotted,dotsep=2pt](8.5,4)(10,4)
\psdot(4,2)
\uput[d](4,2){$v_{2k+1}$}
\psdot(5.5,2)
\uput[d](5.5,2){$v_{2k}$}
\psdot(7,2)
\uput[d](7,2){$v_{2k-1}$}
\psdot(8.5,2)
\uput[d](8.5,2){$v_{2k-2}$}
\psline(3,3)(4,2)
\psline(4,2)(8.5,2)
\psline[linestyle=dotted,dotsep=2pt](8.5,2)(10,2)
\psdot(10,2)
\uput[d](10,2){$v_{k+5}$}
\psdot(11.5,2)
\uput[d](11.5,2){$v_{k+4}$}
\psdot(13,2)
\uput[d](13,2){$v_{k+3}$}
\psdot(14.5,2)
\uput[d](14.5,2){$v_{k+2}$}
\psline(14.5,4)(14.5,2)
\psline(10,2)(14.5,2)
\end{picture}
\caption{Odd Cycle with 2k+1 vertices}
\end{figure}
Let $n=2k+1,k\in \mathbb{Z^+}$ and $C_n=(v_1,v_2,...,v_{2k+1})$ be the given cycle. Consider a vertex $v_1$. Then $v_{k+1}$ and $v_{k+2}$ are its antipodal vertices at a distance $k$ from  $v_1$  and there is no geodesic path from the vertex $v_{k+1}$ and $v_{k+2}$ to any other vertex passing through $v_1$. Hence we omit $v_1$ and the pair $(v_{k+1},v_{k+2})$. Now consider paths of length $\leq k$ passing through $v_1$. There are $k+1-i$ paths joining $v_i$ to vertices from $v_{2k+1}$ to $v_{k+1+i}$is a geodesic path through $v_1$ giving a betweenness centrality 1 to $v_1$. Consider paths from $v_i$ to the vertices $v_{2k+1},v_{2k},\dots,v_{k+i+1}$ passing through $v_1$ for each $i=2,3,\dots,k$. Therefore the betweenness centrality of $v_1$ is $\sum_{i=2}^{k}(k+1-i)=\frac{k(k-1)}{2}=\frac{(n-1)(n-3)}{8}$. Since $C_n$ is vertex transitive, the betweenness centrality of any vertex is given by $\frac{(n-1)(n-3)}{8}$.\\  
\end{proof}
The relative centrality and graph centrality are as follows  \\
\[
  C_{B}{'}{(v)} = \frac{C_B(v)}{Max \,C_B(v)}=\frac{2C_B(v)}{(n-1)(n-2)} = \begin{dcases*}
        \frac{n-2}{4(n-1)}  & if $n$ is even\\
        \frac{n-3}{4(n-2)} & if $n$ is odd
        \end{dcases*}
\]

$C_B(C_n)=0$

\subsection{Betweenness centrality of vertices in circular ladder graphs $CL_n$ }
\begin{figure}[H]
\centering
 \includegraphics[scale=0.5]{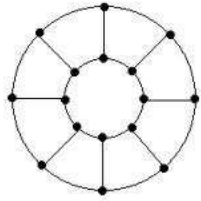}
 \caption{Circular Ladder}
\end{figure}

 \vspace{0cm}
 The circular ladder graph $CL_n$ consists of two concentric $n$-cycles in which each of the $n$ corresponding vertices is joined by an edge. It is a 3-regular simple graph isomorphic to the cartesian product $K_2\times C_n$.
\begin{Theorem}
The betweenness centrality of a vertex in a circular ladder  $CL_n$ is given by
\[
 C_{B}{(v)} =  \begin{dcases*}
       \frac{(n-1)^2+1}{4} & when $n$ is even\\
        \frac{(n-1)^2}{4} & when $n$ is odd
        \end{dcases*}
\]
\end{Theorem}

\begin{proof}
Case$1$: When $n$ is even\\
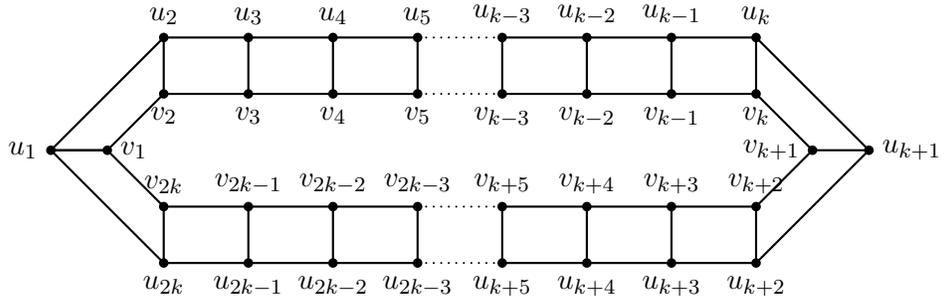
\begin{figure}[H]
\psset{unit=.75}\label{fig7}
\begin{picture}(10000,100)(-70,30)
\psdot(3,3)
\psdot(2,3)
\uput[r](3,3){$v_1$}
\uput[l](2,3){$u_1$}
\psdot(4,4)
\psdot(4,5)
\uput[d](4,4){$v_2$}
\uput[u](4,5){$u_2$}
\psdot(5.5,4)
\psdot(5.5,5)
\uput[d](5.5,4){$v_3$}
\uput[u](5.5,5){$u_3$}
\psdot(7,4)
\psdot(7,5)
\uput[d](7,4){$v_4$}
\uput[u](7,5){$u_4$}
\psdot(8.5,4)
\psdot(8.5,5)
\uput[d](8.5,4){$v_5$}
\uput[u](8.5,5){$u_5$}
\psdot(10,4)
\psdot(10,5)
\uput[d](10,4){$v_{k-3}$}
\uput[u](10,5){$u_{k-3}$}
\psdot(11.5,4)
\psdot(11.5,5)
\uput[d](11.5,4){$v_{k-2}$}
\uput[u](11.5,5){$u_{k-2}$}
\psdot(13,4)
\psdot(13,5)
\uput[d](13,4){$v_{k-1}$}
\uput[u](13,5){$u_{k-1}$}
\psdot(14.5,4)
\psdot(14.5,5)
\uput[d](14.5,4){$v_{k}$}
\uput[u](14.5,5){$u_{k}$}
\psline(3,3)(4,4)
\psline(2,3)(4,5)
\psline(4,4)(8.5,4)
\psline(4,5)(8.5,5)
\psline(10,4)(14.5,4)
\psline(10,5)(14.5,5)
\psline[linestyle=dotted,dotsep=2pt](8.5,4)(10,4)
\psline[linestyle=dotted,dotsep=2pt](8.5,5)(10,5)
\psdot(15.5,3)
\psdot(16.5,3)
\uput[l](15.5,3){$v_{k+1}$}
\uput[r](16.5,3){$u_{k+1}$}
\psline(14.5,4)(15.5,3)
\psline(14.5,5)(16.5,3)
\psdot(4,2)
\psdot(4,1)
\uput[u](4,2){$v_{2k}$}
\uput[d](4,1){$u_{2k}$}
\psdot(5.5,2)
\psdot(5.5,1)
\uput[u](5.5,2){$v_{2k-1}$}
\uput[d](5.5,1){$u_{2k-1}$}
\psdot(7,2)
\psdot(7,1)
\uput[u](7,2){$v_{2k-2}$}
\uput[d](7,1){$u_{2k-2}$}
\psdot(8.5,2)
\psdot(8.5,1)
\uput[u](8.5,2){$v_{2k-3}$}
\uput[d](8.5,1){$u_{2k-3}$}
\psline(3,3)(4,2)
\psline(2,3)(4,1)
\psline(4,2)(8.5,2)
\psline(4,1)(8.5,1)
\psline[linestyle=dotted,dotsep=2pt](8.5,2)(10,2)
\psline[linestyle=dotted,dotsep=2pt](8.5,1)(10,1)
\psdot(10,2)
\psdot(10,1)
\uput[u](10,2){$v_{k+5}$}
\uput[d](10,1){$u_{k+5}$}
\psdot(11.5,2)
\psdot(11.5,1)
\uput[u](11.5,2){$v_{k+4}$}
\uput[d](11.5,1){$u_{k+4}$}
\psdot(13,2)
\psdot(13,1)
\uput[u](13,2){$v_{k+3}$}
\uput[d](13,1){$u_{k+3}$}
\psdot(14.5,2)
\psdot(14.5,1)
\uput[u](14.5,2){$v_{k+2}$}
\uput[d](14.5,1){$u_{k+2}$}
\psline(15.5,3)(14.5,2)
\psline(10,1)(14.5,1)
\psline(10,2)(14.5,2)
\psline(14.5,1)(16.5,3)
\psline(2,3)(3,3)
\psline(4,5)(4,4)
\psline(5.5,5)(5.5,4)
\psline(2,3)(3,3)
\psline(7,5)(7,4)
\psline(8.5,5)(8.5,4)
\psline(10,5)(10,4)
\psline(11.5,5)(11.5,4)
\psline(13,5)(13,4)
\psline(14.5,5)(14.5,4)
\psline(15.5,3)(16.5,3)
\psline(4,2)(4,1)
\psline(5.5,2)(5.5,1)
\psline(7,2)(7,1)
\psline(8.5,2)(8.5,1)
\psline(10,2)(10,1)
\psline(11.5,2)(11.5,1)
\psline(13,2)(13,1)
\psline(14.5,2)(14.5,1)
\end{picture}\vspace{1cm}
\caption{Circular Ladder $CL_{2k}$ }
\end{figure}
Let $n=2k,k\in \mathbb{Z^+}$. $C_{2k}=(u_1,u_2,...,u_{2k})$ be the outer cycle and $C'_{2k}=(v_1,v_2,...,v_{2k})$ be the inner cycle. Consider any vertex say $u_1$ in $C_{2k}$. Then  its betweenness centrality as a vertex in $C_{2k}$ is $\frac{(k-1)^2}{2}$. Now the geodesics from outer vertices  $u_i$ to the inner vertices  $v_1,v_{2k},\dots,v_{k+i}$ for $i=2,\dots,k$ (Fig 3.10) and from  $u_{2k+2-i}$ to $v_1,v_2,\dots,v_{k+2-i}$ for $i=2,\dots,k$ by symmetry, contribute to $u_1$ the betweenness centrality $$2\sum_{i=2}^{k}\Bigl(\frac{1}{i}+\frac{2}{i+1}+\dots+\frac{k+1-i}{k}+\frac{k+2-i}{2k+2}\Bigl)=2\Bigr(\frac{1}{2}+\frac{1+2}{3}+\dots+\frac{1+2+\dots+k-1}{k}+\frac{2+3+\dots+k}{2k+2}\Bigr)$$\\$$=2\sum_{p=2}^{k}\frac{(1+2+\dots+p-1)}{p}+\frac{k(k+1)-2}{2(k+1)}=2\sum_{p=2}^{k }\frac{p-1}{2}+\frac{k(k+1)-2}{2(k+1)}=\frac{k^2}{2}-\frac{1}{k+1}$$\\Again the pair $(u_{k+1},v_1)$ contributes to $u_1$ the betweenness centrality $\frac{2}{2k+2}$. Therefore the betweenness centrality of $u_1$ is given as\\
$$C_B(u_1)=\frac{(k-1)^2}{2}+\frac{k^2}{2}-\frac{1}{k+1}+\frac{1}{k+1}
=\frac{(k-1)^2}{2}+\frac{k^2}{2}=\frac{(2k-1)^2+1}{4}=\frac{(n-1)^2+1}{4}$$

 Case$2$: When $n$ is odd\\
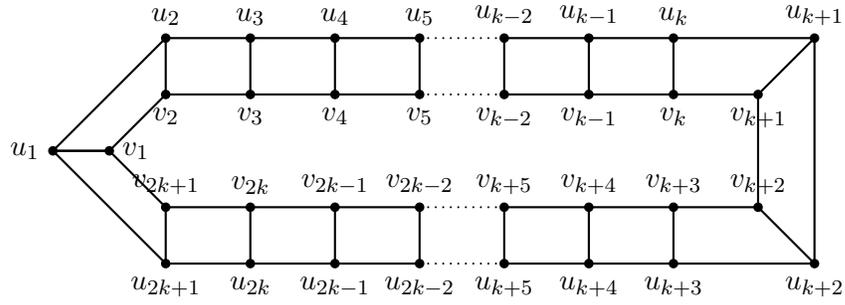
\begin{figure}[H]
\psset{unit=.75}\label{fig7}
\begin{picture}(10000,100)(-70,30)
\psdot(3,3)
\psdot(2,3)
\uput[r](3,3){$v_1$}
\uput[l](2,3){$u_1$}
\psdot(4,4)
\psdot(4,5)
\uput[d](4,4){$v_2$}
\uput[u](4,5){$u_2$}
\psdot(5.5,4)
\psdot(5.5,5)
\uput[d](5.5,4){$v_3$}
\uput[u](5.5,5){$u_3$}
\psdot(7,4)
\psdot(7,5)
\uput[d](7,4){$v_4$}
\uput[u](7,5){$u_4$}
\psdot(8.5,4)
\psdot(8.5,5)
\uput[d](8.5,4){$v_5$}
\uput[u](8.5,5){$u_5$}
\psdot(10,4)
\psdot(10,5)
\uput[d](10,4){$v_{k-2}$}
\uput[u](10,5){$u_{k-2}$}
\psdot(11.5,4)
\psdot(11.5,5)
\uput[d](11.5,4){$v_{k-1}$}
\uput[u](11.5,5){$u_{k-1}$}
\psdot(13,4)
\psdot(13,5)
\uput[d](13,4){$v_{k}$}
\uput[u](13,5){$u_{k}$}
\psdot(14.5,4)
\psdot(15.5,5)
\uput[d](14.5,4){$v_{k+1}$}
\uput[u](15.5,5){$u_{k+1}$}
\psline(3,3)(4,4)
\psline(2,3)(4,5)
\psline(4,4)(8.5,4)
\psline(4,5)(8.5,5)
\psline(10,4)(14.5,4)
\psline(10,5)(15.5,5)
\psline[linestyle=dotted,dotsep=2pt](8.5,4)(10,4)
\psline[linestyle=dotted,dotsep=2pt](8.5,5)(10,5)
\psline(14.5,2)(14.5,4)
\psdot(4,2)
\psdot(4,1)
\uput[u](4,2){$v_{2k+1}$}
\uput[d](4,1){$u_{2k+1}$}
\psdot(5.5,2)
\psdot(5.5,1)
\uput[u](5.5,2){$v_{2k}$}
\uput[d](5.5,1){$u_{2k}$}
\psdot(7,2)
\psdot(7,1)
\uput[u](7,2){$v_{2k-1}$}
\uput[d](7,1){$u_{2k-1}$}
\psdot(8.5,2)
\psdot(8.5,1)
\uput[u](8.5,2){$v_{2k-2}$}
\uput[d](8.5,1){$u_{2k-2}$}
\psline(3,3)(4,2)
\psline(2,3)(4,1)
\psline(4,2)(8.5,2)
\psline(4,1)(8.5,1)
\psline[linestyle=dotted,dotsep=2pt](8.5,2)(10,2)
\psline[linestyle=dotted,dotsep=2pt](8.5,1)(10,1)
\psdot(10,2)
\psdot(10,1)
\uput[u](10,2){$v_{k+5}$}
\uput[d](10,1){$u_{k+5}$}
\psdot(11.5,2)
\psdot(11.5,1)
\uput[u](11.5,2){$v_{k+4}$}
\uput[d](11.5,1){$u_{k+4}$}
\psdot(13,2)
\psdot(13,1)
\uput[u](13,2){$v_{k+3}$}
\uput[d](13,1){$u_{k+3}$}
\psdot(14.5,2)
\psdot(15.5,1)
\uput[u](14.5,2){$v_{k+2}$}
\uput[d](15.5,1){$u_{k+2}$}
\psline(15.5,5)(14.5,4)
\psline(10,1)(15.5,1)
\psline(10,2)(14.5,2)
\psline(2,3)(3,3)
\psline(4,5)(4,4)
\psline(5.5,5)(5.5,4)
\psline(2,3)(3,3)
\psline(7,5)(7,4)
\psline(8.5,5)(8.5,4)
\psline(10,5)(10,4)
\psline(11.5,5)(11.5,4)
\psline(13,5)(13,4)
\psline(15.5,1)(14.5,2)
\psline(4,2)(4,1)
\psline(5.5,2)(5.5,1)
\psline(7,2)(7,1)
\psline(8.5,2)(8.5,1)
\psline(10,2)(10,1)
\psline(11.5,2)(11.5,1)
\psline(13,2)(13,1)
\psline(15.5,1)(15.5,5)
\end{picture}\vspace{.8cm}
\centering
\caption{Circular Ladder $CL_{2k+1}$. }
\end{figure}
Let $n=2k+1,k\in \mathbb{Z^+}$. $C_{2k+1}=(u_1,u_2,...,u_{2k+1})$ be the outer cycle and $C'_{2k+1}=(v_1,v_2,...,v_{2k+1})$ be the inner cycle. Consider any vertex say $u_1$ in $C_{2k+1}$. Then  its betweenness centrality as a vertex in $C_{2k+1}$ is $\frac{k(k-1)}{2}$. Now consider the geodesics  from outer vertices  $u_i$ to the inner vertices  $v_1,v_{2k+1},\dots,v_{k+i+1}$ for $i=2,\dots,k+1$ (Fig 3.11) and from $u_{2k+3-i}$ to  $v_1,v_2,\dots,v_{k+2-i}$ for $i=2,3,\dots,k+1$ which gives a betweenness centrality $$2\Bigl(\frac{1}{i}+\frac{2}{i+1}+\dots+\frac{k+2-i}{k+1}\Bigr)=2\Bigr(\frac{1}{2}+\frac{1+2}{3}+\dots+\frac{1+2+\dots+k}{k+1}\Bigr)$$\\$$=2\sum_{p=2}^{k+1}\frac{(1+2+\dots+p-1)}{p}=2\sum_{p=2}^{k+1}\frac {p-1}{2}=\frac{k(k+1)}{2}$$\\                                                                                               Therefore the betweenness centrality of $u_1$ is given as\\
$$C_B(u_1)=\frac{k(k-1)}{2}+\frac{k(k+1)}{2}=k^2=\frac{(n-1)^2}{4}$$.
\end{proof}
Relative centrality 
\[
 C_{B}{'}{(v)} = \frac{2C_B(v)}{(2n-1)(2n-2)}= \begin{dcases*}
       \frac{(n-1)^2+1}{2(2n-1)(2n-2)}& when n is even\\
       \frac{(n-1)}{4(2n-1)}& when n is odd
        \end{dcases*}
\]
Graph centrality
$$C_B(G)=0$$
\subsection{Betweenness centrality of vertices in hypercubes}
\begin{figure}[h!]
 \centering
  \includegraphics[scale=0.5]{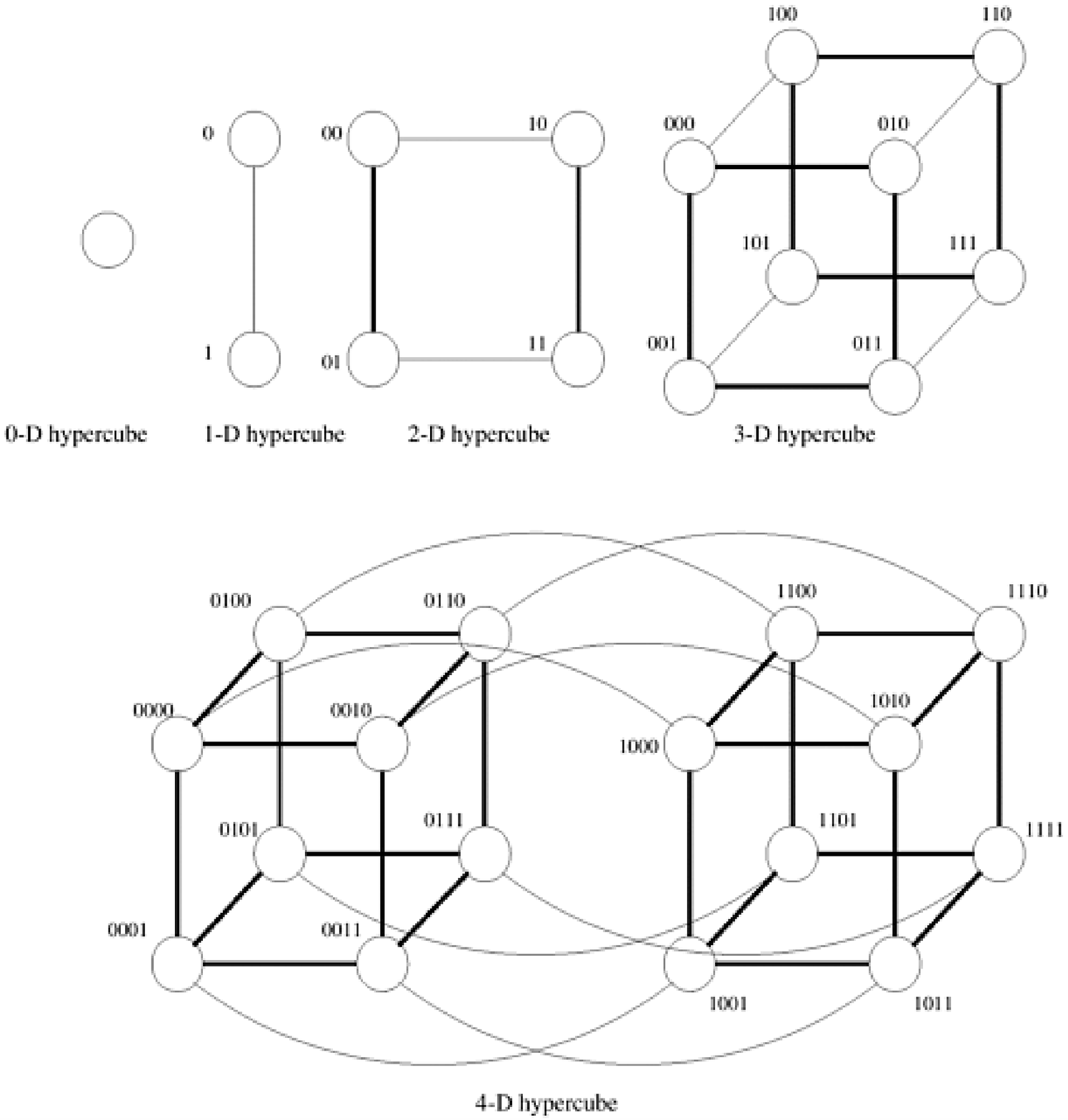}
   \caption{Hypercubes.}
   \label{Fig:HC}
\end{figure}

  The $n$-cube or $n$-dimensional hypercube $Q_n$ is defined recursively by
$Q_1 =K_2$ and
$Q_n= K_2 \times Q_{n-1}$. That is, $Q_n= (K_2)^n$ Harary \cite{harary6graph}.
$Q_n$ is an $n$-regular graph containing $2^n$ vertices and $n2^{n-1}$ edges. Each vertex can be labeled as a string of  $n$  bits $0$ and $1$. Two vertices of  $Q_n$ are adjacent if their binary representations differ at exactly one place (Fig:3.12). The $2^n$ vertices are labeled by the $2^n$ binary numbers from 0 to $2^n-1$. By definition, the length of a path between two vertices $u$ and $v$ is  the number of edges of the path. To reach $v$ from $u$ it suffices to cross successively the vertices whose labels are those obtained by modifying the bits of $u$ one by one in order to transform $u$ into $v$. If $u$ and $v$ differ only in $i$ bits, the distance between $u$ and $v$ denoted by $d(u,v)$, the hamming distance  is $i$ \cite{foldes1977characterization,
saad1988topological}. For example, if $u=(101010)$ and $v=(110011)$ then $d(u,v)=3$.\\ There exists a path of length at most $n$ between any two vertices of $Q_n$. In other words an $n$-cube is a connected graph of diameter $n$. The number of geodesics between $u$ and $v$ is given by the number of permutations  $d(u,v)!$. A hypercube is bipartite and interval-regular. For any two vertices $u$ and $v$, the interval $I(u,v)$ induces a hypercube of dimension $d(u,v)$\cite{mulder1982interval}. 
Another important property of $n$-cube is that it can be constructed recursively from lower dimensional cubes. Consider two identical $(n-1)$-cubes.  Each $(n-1)$-cube has $2^{n-1}$ vertices and each vertex has a labeling of $(n-1)$- bits. Join all identical pairs of the two cubes. Now increase the number of bits in the labels of all vertices by placing 0 in the  $i^{th}$ place of the first cube and 1 in the  $i^{th}$ place of  second cube. Thus we get an $n$-cube with $2^n$ vertices, each vertex having a label of  $n$-bits  and the corresponding vertices of the two $(n-1)$-cubes differ only in the $i^{th}$ bit. This $n$-cube so constructed can  be seen as the union of $n$ pairs of $(n-1)$-cubes differing in exactly one position of bits. Thus the number of $(n-1)$-cubes in an $n$-cube is 2$n$. The operation of splitting an $n$-cube into two disjoint $(n-1)$-cubes so that the vertices of the two $(n-1)$-cubes  are in a one-to-one correspondence  will be referred  to as  $\textit{tearing}$ \cite{saad1988topological}.  Since there are $n$ bits, there are  $n$ directions for $\textit{tearing}$ . In general there are $^nC_k2^{n-k}$ number of $k$-subcubes associated with an $n$-cube. 
\vspace{1cm}
\begin{Theorem}
The betweenness centrality of a vertex in a hypercube $Q_n$ is given by \;$\mathbf{2^{n-2}(n-2)+\frac{1}{2}}$ 
\end{Theorem}
\begin{proof}
The hypercube $Q_n$ of  dimension $n$ is a vertex transitive $n$-regular graph containing $2^n$ vertices. Each vertex can be written as an $n$ tuple of binary digits $0$ and $1$ with adjacent vertices differing in exactly one coordinate. The distance between two vertices $x$ and $y$ denoted by $d(x,y)$ is the number of corresponding coordinates they differ and the number of distinct geodesics between $x$ and $y$ is $d(x,y)!$                                  
 \cite{foldes1977characterization}. Let $\mathbf{0}=(0,0,...,0)$ be the vertex whose betweenness centrality has to be determined. Consider all $k$-subcubes  containing the vertex $\mathbf{0}$ for $2\leq k\leq n$. Each $k$-subcube has vertices with $n-k$ zeros in the $n$ coordinates. Since each $k$-subcube can be distinguished by $k$ coordinates, we consider these $k$ coordinates only. The number of $k$-subcubes containing the vertex $\mathbf{0}$ is $\binom{n}{k}$. The vertex $\mathbf{0}$ lies on a geodesic path joining a pair of vertices if and only if the pair of vertices forms a pair of antipodal vertices of some subcube containing $\mathbf{0}$ 
  \cite{mulder1982interval}. So we consider all  pairs of antipodal vertices excluding the vertex $\mathbf{0}$ and its antipodal vertex in each $k$-subcube containing $\mathbf{0}$.  If a vertex of a $k$-subcube has $r$ ones, then its antipodal vertex has $k-r$ ones.  For any pair of such antipodal vertices there are $k!$ geodesic paths joining them and of that $r!(k-r)!$ paths are  passing through $\mathbf{0}$. Thus each pair contributes $\frac{r!(k-r)!}{k!}$ that is, $ \frac{1}{\binom{k}{r}}$ to the betweenness centrality of $\mathbf{0}$.\\ By symmetry when $k$ is even, the number of distinct pairs of required antipodal vertices are given by $\binom{k}{r}$ for $1\leq r<\frac{k}{2}$ and $\frac{1}{2}\binom{k}{r}$ for $r=\frac{k}{2}$. When $k$ is odd, the number of distinct pairs of required antipodal vertices are given by $\binom{k}{r}$ for $1\leq r\leq \frac{k-1}{2}$. Taking all such pairs of antipodal vertices in a $k$-subcube we get the contribution of betweenness centrality as $\sum_{r=1}^{\frac{k}{2}-1} \binom{k}{r}\frac{1}{\binom{k}{r}}+\frac{1}{2}\binom{k}{\frac{k}{2}}\frac{1}{\binom{k}{\frac{k}{2}}}=\frac{k-1}{2}$, when $k$ is even and $\sum_{r=1}^{\frac{k-1}{2}} \binom{k}{r}\frac{1}{\binom{k}{r}}=\frac{k-1}{2}$ when $k$ is odd. Therefore considering all $k$-subcubes for $2\leq k\leq n$, we get the betweenness centrality of $\mathbf{0}$ as\\
$$C_B(\mathbf{0})=\sum_{k=2}^n \Bigl(\frac{k-1}{2}\Bigr)\binom{n}{k}= \frac{1}{2}\Bigl[\sum_{k=2}^n k\binom{n}{k}-\sum_{k=2}^n \binom{n}{k}\Bigr]=2^{n-2}(n-2)+\frac{1}{2} $$
Therefore for any vertex $v$, \\
$$C_B(\mathbf{v})=\mathbf{2^{n-2}(n-2)+\frac{1}{2}} $$
\end{proof}
The relative centrality and graph centrality are as follows\\$$C_{B}{'}{(v)}=\frac{2C_B(v)}{(2^n-1)(2^n-2)}=\frac{2^{n-1}(n-2)+1}{(2^n-1)(2^n-2)}$$     
$$C_B(G)=0$$
\section{Conclusion}
Betweenness centrality is known to be a useful metric for graph analysis. When compared to other centrality measures, computation of betweenness centrality  is rather difficult as it involves calculation of the shortest paths between all pairs of vertices in a graph. This study of betweenness centrality  can  be extended to larger classes of graphs and for edges also.
\bibliographystyle{plain}
\bibliography{sunilbiblio}
\end{document}